\documentclass[11pt]{article}

\usepackage{graphicx,tikz,color,url}
\usepackage{amsmath,amssymb,latexsym}
\usepackage{mathtools}
\usepackage{xspace}

\newtheorem{theorem}{Theorem}[section]

\newtheorem{lemma}[theorem]{Lemma}
\newtheorem{proposition}[theorem]{Proposition}

\newtheorem{corollary}[theorem]{Corollary}

\newcommand{\proof}{\noindent{\bf Proof.\ }}
\newcommand{\qed}{\hfill $\square$ \bigskip}

\newcommand{\mut}{\mu_{\rm t}}
\newcommand{\Int}{V_\mathrm{Int}}
\newcommand{\Ext}{V_\mathrm{Ext}}
\newcommand{\Hull}{\mathit{hull}}

\definecolor{auburn}{rgb}{0.43, 0.21, 0.1}
\definecolor{blush}{rgb}{0.87, 0.36, 0.51}

\definecolor{mycolor}{rgb}{0.0, 0.6, 0.5}

\begin{document}

\title{Mutual-visibility in strong products of graphs \\ via total mutual-visibility}

\author{
Serafino Cicerone $^{a}$\thanks{Email: \texttt{serafino.cicerone@univaq.it}}
\and
Gabriele Di Stefano $^{a}$\thanks{Email: \texttt{gabriele.distefano@univaq.it}}
\and
Sandi Klav\v zar $^{b,c,d}$\thanks{Email: \texttt{sandi.klavzar@fmf.uni-lj.si}}
\and Ismael G. Yero $^{e}$\thanks{Email: \texttt{ismael.gonzalez@uca.es}}
}

\maketitle

\begin{center}
	$^a$ Department of Information Engineering, Computer Science, and Mathematics,
	     University of L'Aquila, Italy \\
	\medskip

	$^b$ Faculty of Mathematics and Physics, University of Ljubljana, Slovenia\\
	\medskip
	
	$^c$ Institute of Mathematics, Physics and Mechanics, Ljubljana, Slovenia\\
	\medskip
	
	$^d$ Faculty of Natural Sciences and Mathematics, University of Maribor, Slovenia\\
	\medskip
	
	$^e$ Departamento de Matem\'aticas, Universidad de C\'adiz, Algeciras Campus, Spain \\

\end{center}

\begin{abstract}
Let $G$ be a graph and $X\subseteq V(G)$. Then $X$ is a mutual-visibility set if each pair of vertices from $X$ is connected by a geodesic with no internal vertex in $X$. The mutual-visibility number $\mu(G)$ of $G$ is the cardinality of a largest mutual-visibility set. In this paper, the mutual-visibility number of strong product graphs is investigated. As a tool for this, total mutual-visibility sets are introduced. Along the way, basic properties of such sets are presented. The (total) mutual-visibility number of strong products is bounded from below in two ways, and determined exactly for strong grids of arbitrary dimension. Strong prisms are studied separately and a couple of tight bounds for their mutual-visibility number are given.
\end{abstract}


\noindent
AMS Subj.\ Class.\ (2020): 05C12, 05C76

\maketitle

\section{Introduction}

Let $G = (V(G), E(G))$ be a connected and undirected graph, and $X\subseteq V(G)$ a subset of the vertices of $G$. If $x,y\in V(G)$, then we say that $x$ and $y$ are {\em $X$-visible}, if there exists a shortest $x,y$-path whose internal vertices are all not in $X$. $X$ is a \emph{mutual-visibility set} if its vertices are pairwise $X$-visible. The cardinality of a largest mutual-visibility set is the \emph{mutual-visibility number} of $G$, and it is denoted by $\mu(G)$. Each largest mutual-visibility set is also called {\em $\mu$-set} of $G$.

These concepts were introduced by Di Stefano in~\cite{DiStefano-2022}. They were in particular motivated by the significance that mutual-visibility properties play within problems that arise in mobile entity models. Some of the numerous works that deal with such models are~\cite{aljohani-2018a, bhagat-2020, diluna-2017, poudel-2021}. Mutual-visibility sets in graphs are in a way dual to general position sets in graphs, the latter concepts being widely investigated in the last years~\cite{klavzar-2021, manuel-2018, patkos-2020, tian-2021, ullas-2016}.

Among other results, it was proved in~\cite{DiStefano-2022} that the decision problem concerning the mutual-visibility number is NP-complete and the invariant was determined for several classes of graphs including block graphs, grids, and cographs. The research was continued in~\cite{Cicerone-2022+} emphasizing on Cartesian products and graphs $G$ with $\mu(G) = 3$. Interestingly, determining the mutual-visibility number of the Cartesian product of two complete graphs turns out to be equivalent to a case of the celebrated Zarankiewicz's problem which is a long-standing open combinatorial problem. Continuing the investigation of the mutual-visibility in graph products, we investigate in this paper strong products.

In the next section, we introduce the necessary concepts and recall some known results. Then, in Section~\ref{sec:mut}, we introduce total mutual-visibility sets which turned out to be useful for the investigation of mutual-visibility sets in strong products, and give some basic properties of total mutual-visibility sets. In the subsequent section, we first bound from below the (total) mutual-visibility number of strong products. Then we determine the mutual-visibility number for the strong grids of arbitrary dimension which shows the tightness of the lower bound. In addition, we find families of strong product graphs for which the bound is not tight and complete the section with another lower bound. In Section~\ref{sec:strong-prisms} we focus on strong prisms where we give a couple of tight bounds for the mutual-visibility number. We conclude our exposition with several open problems and directions for further investigation.

\section{Preliminaries}\label{sec:preliminaries}

Since two vertices from different components of a graph are not mutually visible, all graphs in the paper are connected unless stated otherwise.

For a natural number $n$, we set $[n] = \{1,\ldots, n\}$. Given a graph $G = (V(G), E(G))$, its order will be denoted by $n(G)$. The distance function $d_G$ on a graph $G$ is the usual shortest-path distance.
The subgraph $G'$ is \emph{convex} if, for every two vertices of $G'$, every shortest path in $G$ between them lies completely in $G'$.  The \emph{convex hull} of $V'\subseteq V(G)$, denoted as $\Hull(V')$, is defined as the smallest convex subgraph containing $V'$. A \emph{universal vertex} is a vertex that is adjacent to all other vertices of the graph.

The {\em degree}, $\deg_G(x)$, of a vertex $x$ is the number of its neighbors. If $X\subseteq V(G)$, then $\overline{X}$ denotes the complement of $X$, that is the set containing all vertices of $G$ not in $X$. Moreover, $G[X]$ denotes the subgraph of $G$ induced by $X$, that is the maximal subgraph of $G$ with vertex set $X$. The subgraph of $G$ induced by $\overline{X}$ is denoted by $G-X$, and by $G-v$ when $X=\{v\}$. Two vertices $u$ and $v$ are \emph{false twins} if $uv\not\in E(G)$ and $N_G(u) = N_G(v)$, where $N_G(x)$ is the open neighborhood of $x$, and are \emph{true twins} if $uv\in E(G)$ and $N_G[u] = N_G[v]$, where $N_G[x]$ is the closed neighborhood of $x$. Vertices are {\em twins} if they are true or false twins. Adding a new vertex to a graph $G$ that is a true/false twin of an existing vertex of $G$ is an operation called \emph{splitting}. Another one-vertex extending operation is that of attaching a \emph{pendant vertex}, which is a vertex connected by a single edge to an existing vertex of the graph.

A graph is a \emph{block graph} if every block (i.e., a maximal 2-connected component) is a clique. Block graphs can be generated by using true twins and pendant vertices. Notice that the connected block graphs are exactly the graphs in which there is a unique induced path connecting every pair of vertices.

A graph is called \emph{cograph} whenever it is obtained by a sequence of splittings starting from $K_1$. From this generative definition, it follows a useful structural property. Let $G$ be a cograph, and let $v_1$ be the starting vertex for a sequence of splitting operations that build $G$. If $G$ is connected, the first operation must be a true twin of $v_1$ (that produces $v_2$ adjacent to $v_1$). Let $V_1 = \{v_1\}$ and $V_2 = \{v_2\}$. Now, for each further vertex $v$ which must be added to build $G$, if $v$ is a twin of a vertex in $V_1$ ($V_2$, respectively), then add it to $V_1$ (to $V_2$, respectively). We obtain that $V(G)$ can be partitioned into $V_1$ and $V_2$ where $v'v''\in E(G)$ for each $v'\in V_1$ and $v''\in V_2$.

Cographs include complete split graphs and complete $k$-partite graphs. A graph is a {\em complete split graph} if it can be partitioned into an independent set and a clique such that every vertex in the independent set is adjacent to every vertex in the clique. A {\em $k$-partite graph} (alias $k$-chromatic graph) is a graph whose vertices are (or can be) partitioned into $k$ different independent sets; hence, a complete $k$-partite graph is a $k$-partite graph in which there is an edge between every pair of vertices from different independent sets.

The {\em strong product} $G\boxtimes H$ of graphs $G$ and $H$ has vertex set $V(G\boxtimes H) = V(G)\times V(H)$, with vertices $(g,h)$ and $(g',h')$ being adjacent in $G\boxtimes H$ if either $gg'\in E(G)$ and $h=h'$, or $g=g'$ and $hh'\in E(H)$, or $gg'\in E(G)$ and $hh'\in E(H)$, see~\cite{Hammack-2011}. A {\em $G$-layer}   through a vertex $(g,h)$ is the subgraph of $G\boxtimes H$ induced by the vertices $\{(g',h):\ g'\in
V(G)\}$. Analogously {\em $H$-layers} are defined.

Finally, we recall the following result which is implicitly used throughout the paper.

\begin{proposition}\label{prop:sp-distance}
{\rm \cite[Proposition 5.4]{Hammack-2011} }
If $(g, h)$ and $(g', h')$ are vertices of a strong product $G\boxtimes H$, then
$$ d_{G\boxtimes H} ((g, h), (g', h')) = \max\{d_G(g, g'), d_H (h, h')\}.$$
\end{proposition}

\section{Total mutual-visibility}
\label{sec:mut}

The following definition introduces a variation of mutual-visibility. It will be useful to provide bounds on the mutual-visibility number of strong product graphs, although we consider that the concept might be also of independent interest.

If $G$ is a graph and $X\subseteq V(G)$, then $X$ is a {\em total mutual-visibility set} of $G$ if every pair of vertices $x$ and $y$ of $G$ is $X$-visible. The term ``total" comes from observing that if $X$ is a total mutual-visibility set of $G$, then for every pair $x,y\in V(G)$ there exists a shortest $x,y$-path whose internal vertices are all not in $X$. The cardinality of a largest total mutual-visibility set of $G$ is the {\em total mutual-visibility number} of $G$ and is denoted by $\mut(G)$. Notice that there could be graphs $G$ which do not contain total mutual-visibility sets, for such situations we set $\mut(G) = 0$. For the sake of brevity, we say that $X$ is a {\em $\mut(G)$-set} (or $\mut$-set if we are not interested in the graph) if it is a total mutual-visibility set such that $|X| = \mut(G)$.

\medskip
Clearly, every total mutual-visibility set is a mutual-visibility set, hence we have the following inequality
\begin{equation}\label{eq:mut-bounds}
0\le \mut(G) \le \mu(G).
\end{equation}
In the following we show that such bounds can actually be achieved by the total mutual-visibility number. Concerning the lower bound of~\eqref{eq:mut-bounds}, it can be easily checked that $\mut(C_n) = 0$ for $n\ge 5$. The variety of graphs with this property appears to be large as the next result confirms.

\begin{proposition}
\label{prop:cover-with-convex}
Let $G$ be a graph. If $V(G) = \bigcup_{i=1}^k V_i$, where $G[V_i]$ is a convex subgraph of $G$ and $\mut(G[V_i]) = 0$ for each $i\in [k]$, then $\mut(G) = 0$.
\end{proposition}
\begin{proof}
Suppose on the contrary that $G$ contains a total mutual-visibility set $X$ with $|X| \ge 1$. Select an arbitrary vertex $x\in X$. Then there exists an $i\in [k]$ such that $x\in V_i$. Hence clearly, $|X\cap G[V_i]| \ge 1$. However, since $G[V_i]$ is convex, we get that $X\cap G[V_i]$ is a total mutual-visibility set of  $G[V_i]$, a contradiction to the assumption $\mut(G[V_i]) = 0$.
\qed
\end{proof}

In what follows we show that there also exist graphs $G$ with $\mut(G) = 0$ such that they belong to well-known graph classes and they are not covered by the Proposition~\ref{prop:cover-with-convex}. To this end, recall that a {\em cactus graph} is a graph whose blocks are cycles and/or complete graphs $K_2$. Figure~\ref{fig:cactus} shows four examples of cactus graphs.

\begin{proposition}
\label{prop:cactus}
Let $G$ be a cactus graph. Then $\mut(G)= 0$ if and only if $G$ has minimum degree $2$ and for each cycle $C$ in $G$ with $n(C)\leq 4$ all the vertices in $C$ have degree at least $3$ in $G$.
\end{proposition}
\begin{proof}
($\Leftarrow$)
Assume that $G$ does not contain pendant vertices and that for each cycle $C$ of $G$, either $n(C)\le 4$ and each vertex in $C$ has degree at least $3$, or $n(C)\ge 5$.
Suppose now $\mut(G)>0$ and consider any total mutual-visibility set $X$ of $G$ with $|X|\ge 1$ and let $v\in X$.
If $v$ does not belong to any cycle of $G$, since there are no pendant vertices, then $v$ must have at least two neighbors and such neighbors are not $X$-visible, which is not possible. Thus, we may consider $v$ belongs to a cycle $C$. If  $n(C)\ge 5$, then the two neighbors of $v$ belonging to $C$ are not $X$-visible. If $n(C)\le 4$ and each vertex in $C$ has degree at least 3, then again there must exist a pair of neighbors of $v$ which are non $X$-visible, a contradiction again. Hence $\mut(G)=0$ must hold.

($\Rightarrow$)
It can be readily observed that each pendant vertex of $G$ forms a total mutual-visibility set of $G$. Thus, $G$ has minimum degree $2$, since $\mut(G)= 0$. Moreover, if $C$ is a cycle in $G$ such that $n(C)\le 4$ and there exists $v\in V(C)$ with $\deg_G(v)=2$, then the set $\{v\}$ is a total mutual-visibility set of $G$, which is not possible. Therefore, the second claim follows.
\qed
\end{proof}

\begin{figure}[ht!]
\begin{center}
\includegraphics[height=4cm]{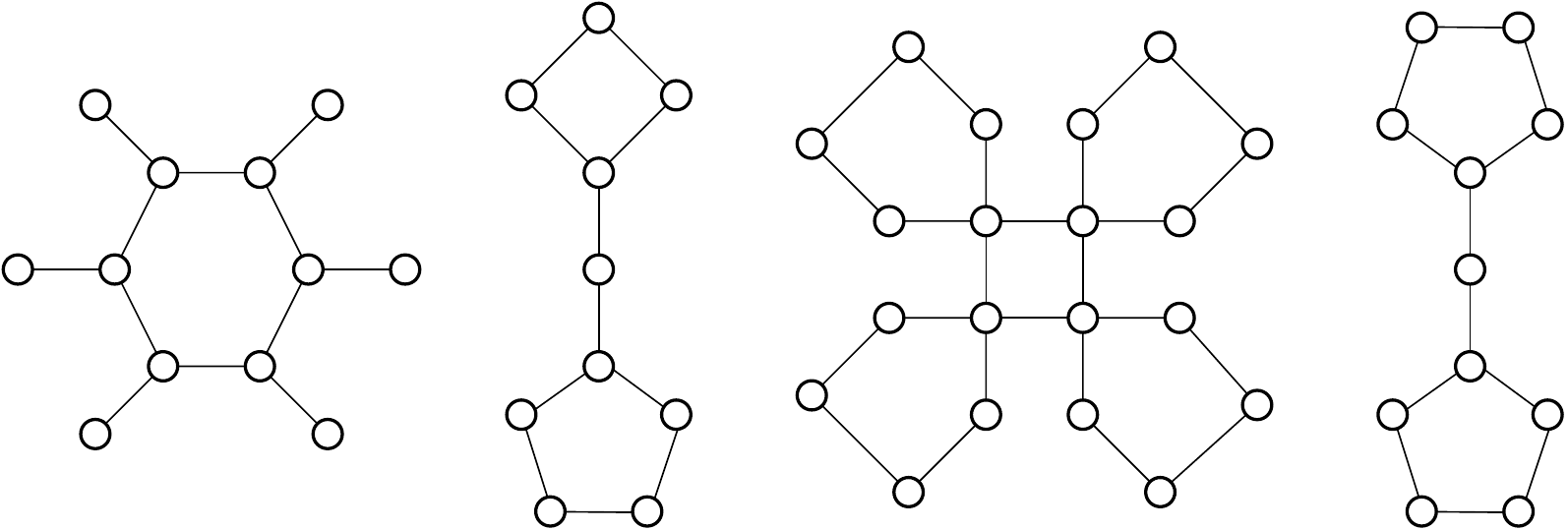}
\caption{\small Some cactus graphs. The first two on the left do not fulfil the conditions of Proposition~\ref{prop:cactus}, and hence their total mutual-visibility number is greater than zero.}
\label{fig:cactus}
\end{center}
\end{figure}

As an application of this lemma, consider Figure~\ref{fig:cactus}. From the left, the first two cactus graphs have total mutual-visibility number greater than zero since they both do not fulfill the conditions of the above lemma. On the contrary, the other two graphs have total mutual-visibility number equal to zero. Moreover, notice that among the cactus graphs it is possible to find infinitely many graphs $G$ with $\mut(G) = 0$ which are not covered by Proposition~\ref{prop:cover-with-convex}. For instance, if $G$ is a cactus graph with minimum degree at least $2$, girth at least $5$, and contains at least one path of length at least $2$ whose edges lie in no cycle, then $\mut(G) = 0$ but $G$ might not admit a proper convex cover as in Proposition~\ref{prop:cover-with-convex}. The rightmost graph in Figure~\ref{fig:cactus} is an example.

\medskip
Concerning the upper bound in~\eqref{eq:mut-bounds}, we introduce the following definition. A graph $G$ is a {\em $(\mu, \mut)$-graph} if $\mu(G) = \mut(G)$.

\begin{proposition}\label{prop:mu-perfect}
Block graphs (and hence trees and complete graphs) and graphs containing a universal vertex are all $(\mu, \mut)$-graphs.
\end{proposition}

\begin{proof}
If $G$ is a complete graph, then $\mu(G) = \mut(G) = n(G)$. If $G$ is not complete and has a universal vertex, then it can be easily observed that $\mu(G) = \mut(G) = n(G)-1$.

Assume that $G$ is a block graph. From~\cite[Theorem 4.2]{DiStefano-2022} we know that if $G$ is a block graph and $X$ the set of its cut-vertices, then $V(G)\setminus X$ is a $\mu$-set of $G$ and hence $\mu(G) = |V(G)\setminus X|$. We now show that $V(G)\setminus X$ is also a $\mut$-set of $G$. To this end, let us first observe that (1) each vertex in $V(G)\setminus X$ is adjacent to a vertex in $X$ and that (2) $G[X]$ is a convex subgraph of $G$. Hence, every $x,y\in V(G)$ are $(V(G)\setminus X)$-visible regardless their membership to $V(G)\setminus X$. This proves that $V(G)\setminus X$ is also a $\mut$-set of $G$. \qed
\end{proof}

In the following, we characterize those cographs which are $(\mu, \mut)$-graphs. To this aim, we first recall a result from~\cite{DiStefano-2022}.

\begin{lemma}\label{lem:enabling}
{\rm \cite[Lemma 4.8]{DiStefano-2022} }
Given a graph $G$, then $\mu(G)\ge n(G)-1$ if and only if there exists a vertex $v$ in $G$ adjacent to each vertex $u$ in $G-v$ such that $\deg_{G-v} (u) < n(G) - 2$.
\end{lemma}

In what follows, any vertex $v$ of $G$ fulfilling the condition in the above lemma will be called \emph{enabling}.

\begin{proposition}\label{prop:mu-perfect-cographs}
A cograph $G$ is a $(\mu, \mut)$-graph if and only if it has a universal vertex or no enabling vertices.
\end{proposition}

\begin{proof}
($\Leftarrow$) If $G$ has a universal vertex, then clearly $\mut(G)=\mu(G)$.
If $G$ has no enabling vertices, then $\mu(G)< n(G)-1$ by Lemma~\ref{lem:enabling}. Since $\mu(G)\geq n(G)-2$ by~\cite[Theorem 4.11]{DiStefano-2022}, we get $\mu(G) = n(G)-2$.  According to the structural property of cographs recalled in Section~\ref{sec:preliminaries}, the vertices of $G$ can be partitioned into two sets $V_1$ and $V_2$ such that each vertex in $V_1$ is adjacent to each vertex of $V_2$. If $v_1$ ($v_2$, respectively) is an arbitrary vertex in $V_1$ ($V_2$, respectively), then it can be easily observed that $X=V(G) \setminus \{v_1,v_2\}$ is a total mutual-visibility set. Hence, $\mut(G)=\mu(G)= n(G)-2$.

($\Rightarrow$) We show that $G$ is not a $(\mu, \mut)$-graph by assuming that $G$ has an enabling vertex $v$ but no universal vertices. In this case, $V(G)$ can be partitioned in three sets: $A=\{v\}$, $B$ the set of neighbors of $v$, and $C$ which contains all the remaining vertices. Notice that $C$ must be not empty otherwise $v$ would be a universal vertex, against the hypothesis.  By definition of enabling vertex, $B$ contains all the vertices $u\in G$ such that $\deg_{G-v}(v) < n(G)-2$. This implies that for each $u\in C$ we have $\deg_{G-v}(u) \ge n(G)-2$. As a consequence, we have that (1) $G[C]$ is a clique, and (2) $bc \in E(G)$ for each $b\in B$ and $c\in C$. Then $B\cup C$ is a mutual-visibility set and hence $\mu(G)\geq n(G)-1$. As $G$ has no universal vertices, $\mu(G)=n(G)-1$.

We now show that $\mut(G)$ cannot be equal to $n(G)-1$. In fact, let $u\in V(G)$ and assume that $X= V(G)\setminus \{u\}$ is a $\mut$-set. Clearly, $u\ne v$ because $v$ is not $X$-visible with vertices in $C$. Vertex $u$ cannot be in $B$ since it is not a universal vertex and so there is a vertex $w \in B$ such that $uw\not \in E(G)$. But then $u$ and $w$ are not $X$-visible. Finally, $u$ cannot be in $C$, because in this case $u$ and $v$ are not $X$-visible.
\qed
\end{proof}

The following result is a straightforward consequence of the characterization provided by Proposition~\ref{prop:mu-perfect-cographs}.

\begin{corollary}
Complete split graphs and complete $k$-partite graphs ($k \geq 2$) with at least three vertices in each partition are $(\mu, \mut)$-graphs.
\end{corollary}

Observe that since $\mu(C_n) = 3$ and $\mut(C_n) = 2$ for $n\le 4$, the inequality $\mut(G) \le \mu(G)$ can be strict. Moreover, even if the equality is attained, it can happen that some $\mu$-sets are $\mut$-sets but some are not. For an example consider the graph from Figure~\ref{fig:mu-sets}.

\begin{figure}[ht!]
\begin{center}
\includegraphics[height=1.9cm]{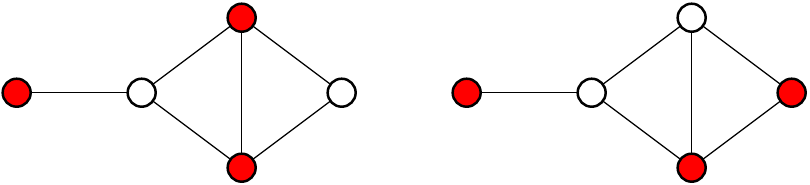}
\caption{\small A graph $G$ with two $\mu$-sets (represented by red vertices). On the right-hand side a $\mu$-set which is also a $\mut$-set is shown, while on the left-hand side the $\mu$-set is not a $\mut$-set (the pair of vertices not in the $\mu$-set are not visible).}
\label{fig:mu-sets}
\end{center}
\end{figure}

\section{Mutual-visibility in strong products}

In this section, we show how the total mutual-visibility of factor graphs can be used to provide lower bounds for the mutual-visibility number of their strong products. To this end, we consider the following refinement of the total mutual-visibility. We say that a total mutual-visibility set $S$ of a graph $G$ is {\em feasible} if any two adjacent vertices $x,y\in S$ have a common neighbor $z\notin S$. Figure~\ref{fig:mu-sets} (right side) shows a graph that admits a feasible $\mut$-set.
Note that there could be graphs having no feasible total mutual-visibility set, say $C_4$. Moreover, if $S$ is a feasible total mutual-visibility set of $G$ and $G$ is not trivial, then it has cardinality at most $|V(G)|-1$. Also, if $S$ is an independent total mutual-visibility set, then it is clearly feasible.

\begin{theorem}\label{thm:mut-lb}
If $S_G$ and $S_H$ are feasible total mutual-visibility sets of the non-trivial graphs $G$ and $H$, respectively, then there exists a feasible total mutual-visibility set $S$ of 
$G\boxtimes H$ such that
$$ |S| \ge |S_G| n(H) + |S_H| n(G) - |S_G| \cdot |S_H|\,.$$
In particular, if both $G$ and $H$ admit feasible $\mut$-sets, then
$$\mut(G\boxtimes H) \ge \mut(G) n(H) + \mut(H) n(G) - \mut(G) \mut(H)\,.$$
\end{theorem}


%
\begin{proof}
Let $S = (V(G)\times V(H)) \setminus (\overline{S_G} \times \overline{S_H})$; see Figure~\ref{fig:prod} for an example of the construction of $S$.

\begin{figure}[ht!]
\begin{center}
\includegraphics[height=4.0cm]{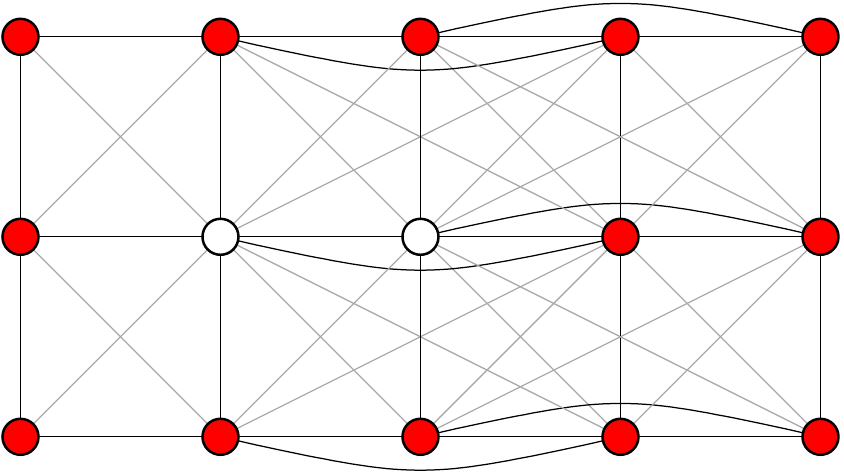}
\caption{\small A representation of $G\boxtimes P_3$, where $G$ is the graph in the right side of Figure~\ref{fig:mu-sets}. The represented $\mut$-set is that defined by Theorem~\ref{thm:mut-lb}.}
\label{fig:prod}
\end{center}
\end{figure}

In the following we first prove that $S$ is a total mutual-visibility set of $G\boxtimes H$, and then we observe that $S$ is feasible. Let $(g,h)$ and $(g',h')$ be arbitrary but distinct vertices from $V(G\boxtimes H)$. Consider first the case in which $d_G(g,g')\ge 2$ and $d_H(h,h')\ge 2$. Regardless the membership of $g,g'$ to $S_G$, since $S_G$ is a total mutual-visibility set of $G$ there exists a shortest $g,g'$-path $P_G$ in $G$ such that no internal vertex of $P_G$ is in $S_G$. Let the consecutive vertices of $P_G$ be $g=g_0, g_1, \ldots, g_k=g'$, with $k\ge 1$ since $d_G(g,g')\ge 2$.
Similarly, there is a shortest $h,h'$-path $P_H$ in $H$ such that no internal vertex of $P_H$ is in $S_H$. Let the consecutive vertices of $P_H$ be $h=h_0, h_1, \ldots, h_\ell=h'$, with $\ell\ge 1$ since $d_H(h,h')\ge 2$.
Assume without loss of generality that $\ell\le k$. Then the vertices
%
%
$$(g,h) = (g_0,h_0), (g_1,h_1), \ldots, (g_{\ell-1}, h_{\ell-1}), (g_{\ell}, h_{\ell-1}), \ldots, (g_{k-1}, h_{\ell-1})(g_k, h_\ell) = (g',h')$$
induce a shortest $(g,h),(g',h')$-path $Q$ in $G\boxtimes H$. Clearly, no internal vertex of $Q$ is in $S$.

If $d_G(g,g')=1$ and $d_H(h,h')=1$, then $(g,g')$ and $(h,h')$ are adjacent and we are done. Thus, we may assume that either $d_G(g,g')\ge 2$ or $d_H(h,h')\ge 2$.

Assume that $d_G(g,g')=1$ and $d_H(h,h')\ge 2$. Consider again the $h,h'$-path $P_H$ defined as above. If $g\notin S_G$, then the vertices
$$(g,h) = (g,h_0), (g,h_1), \ldots, (g, h_{\ell-1}), (g', h_\ell) = (g',h')$$
induce a shortest $(g,h),(g',h')$-path $Q'$ in $G\boxtimes H$, such that no internal vertex of $Q'$ is in $S$. 
On the other hand, if $g\in S_G$ and $g'\in S_G$, then since $S_G$ is a feasible total mutual-visibility set of $G$, the vertices $g,g'$ have a common neighbor $w\notin S_G$. If $g\in S_G$ and $g'\not\in S_G$, then we set $w=g'$. Thus, the vertices
$$(g,h) = (g_0,h_0), (w,h_1), \ldots, (w, h_{\ell-1}), (g', h_\ell) = (g',h')$$
induce a shortest $(g,h),(g',h')$-path $Q''$ in $G\boxtimes H$, such that no internal vertex of $Q''$ is in $S$. Symmetrically, if $d_G(g,g')\ge 2$ and $d_H(h,h')=1$, we obtain analogous conclusions to the ones above.

Consider now the remaining case in which $g=g'$ or $h=h'$ (but not both). By the commutativity of the strong product, we may without loss of generality assume $h= h'$ (and hence $g\neq g'$). Let $g = g_0, g_1, \ldots, g_k = g'$ be a shortest $g,g'$-path in $G$, with $k\ge 1$, such that each internal vertex (if any) is not in $S_G$. If $k=1$, then $(g,h)(g',h')\in E(G\boxtimes H)$ and there is nothing to prove. Assume now $k\ge 2$. 
%
Since  $H$ is non trivial and $S_H$ is feasible, it holds $|S_H| < n(H)$. If $h\not\in S_H$, then 
$$(g,h) = (g_0,h), (g_1,h), (g_2, h), \ldots, (g_{k-1}, h), (g_k,h') = (g',h')$$ trivially shows that $(g,h)$ and $(g',h')$ are mutually visible. Instead, if $h\in S_H$, since $S_H$ is feasible, then there exists a vertex $z\notin S_H$ such that $hz\in E(H)$. 
Consider the path $Q'$ induced by the sequence of vertices
$$(g,h) = (g_0,h), (g_1,z), (g_2, z), \ldots, (g_{k-1}, z), (g_k,h') = (g',h')\,.$$
The length of $Q'$ is $k$, hence $Q'$ is a shortest $(g,h),(g',h')$-path. Moreover, as $z\not\in S_H$ we get that each internal vertex of $Q$ does not belong to $S$. 

According to all the analyzed cases, we have shown that the set $S$ is a total mutual-visibility set. Since
\begin{align*}
|S| & = n(G)n(H) - (n(G) - |S_G|)(n(H) - |S_H|) \\
& = |S_G| n(H) + |S_H| n(G) - |S_G|\cdot |S_H|
\end{align*}
we are done with the first inequality. When $S_G$ and $S_H$ are $\mut$-sets and $G$ and $H$ are non-complete graphs, the second inequality follows directly from the first one.

To conclude the proof, we now show that $S$ is feasible, that is, given two adjacent vertices $(g,h),(g',h')\in S$, we prove that there exists a third vertex not in $S$ which is adjacent to both $(g,h)$ and $(g',h')$. 

Consider first the case when $g\ne g'$ and $h\ne h'$. By the definition of the strong product, $gg'\in E(G)$ and $hh'\in E(H)$. We now claim that there exists a clique $K_G$ of cardinality $2$ or $3$ as a subgraph of $G$ that contains $g$, $g'$, and a vertex $g''$  (which may coincide with $g$ or $g'$) not belonging to $S_G$. If both $g$ and $g'$ lie in $S_G$, then since $S_G$ is feasible, there exists a vertex $g''\not\in S_G$ such that $g$, $g'$, and $g''$ induce a clique of order 3. If $g\not\in S_G$ or $g'\not\in S_G$, then the property trivially holds, since $gg'\in E(G)$ induces a clique of order $2$. We have thus proved that claim, that is, we have a clique $K_G$ (of order 3 or 2) with a vertex $g''$ not in $S_G$. By the commutativity of the strong product, there also exists a clique $K_H$ as a subgraph of $H$ that contains $h$, $h'$, and a vertex $h''$ not belonging to $S_H$ (which may coincide with $h$ or $h'$).
Since the strong product of $K_G$ and $K_H$ induces a clique of $G\boxtimes H$, the vertex $(g'',h'')$ is adjacent to both $(g,h)$ and $(g',h')$. This part of the proof is concluded by observing that $(g'',h'')$ is not in $S$.

The second case to consider is when $g=g'$ or $h=h'$, we may  without loss of generality assume $h=h'$. As in the previous case, there exists a clique $K_G$ of cardinality $2$ or $3$ as a subgraph of $G$ that contains $g$, $g'$, and a vertex $g''\not\in S_G$  (which may coincide with $g$ or $g'$). 
Concerning $H$, since $h=h'$, there exists a clique $K_H$ of cardinality $2$ or $1$ as a subgraph of $H$ that contains $h$ and a vertex $h''$  (which may coincide with $h$) not belonging to $S_H$. 
If $h\not\in S_H$, then $V(K_H)=\{h\}$ and we can set $h''=h$. If $h\in S_H$, then given any neighbor $\bar{h}$ of $h$, either  $\bar{h}\not\in S_H$ (so $V(K_H)=\{h,\bar{h}\}$ and $h''=\bar{h}$) or, since $S_H$ is feasible, there exists $h''\not\in S_H$ adjacent to both $h$ and $\bar{h}$ (so $V(K_H)=\{h,h''\}$). We can again conclude that $(g'',h'')\not\in S$ and $(g'',h'')$ is adjacent to both $(g,h)$ and $(g',h')$.
%
%
\qed
\end{proof}

Of course, when both $G$ and $H$ are $(\mu, \mut)$-graphs that admit feasible $\mut$-sets, the lower bound expressed by Theorem~\ref{thm:mut-lb} can be reformulated as follows:
\begin{equation}\label{eq:mut-lb-perfect}
\mu(G\boxtimes H) \ge \mut(G\boxtimes H) \ge \mu(G)n(H)   + \mu(H)n(G) -  \mu(G) \mu(H)\,.
\end{equation}

Theorem~\ref{thm:mut-lb} extends to an arbitrary number of factors as follows.

\begin{corollary}\label{cor:mut-lb-recursive}
Let $H_k = G_1\boxtimes G_2\boxtimes\cdots \boxtimes G_k$, $k\ge 2$. If $G_i$ is a non-complete graph that admits a feasible $\mut$-set for each $1\le i\le k$, then
$$\mut(H_k) \ge \prod_{i=1}^k n(G_i) - \prod_{i=1}^k (n(G_i) - \mut(G_i)).$$
\end{corollary}

\begin{proof}
For each $1\le i\le k$, let $X_i$ be a feasible $\mut(G_i)$-set. Let
$$S_k = (V(G_1)\times \cdots \times V(G_k)) \setminus
           (\overline{X_1} \times \cdots \times \overline{X_k}).$$
We prove that $S_k$ is a total mutual-visibility set of $H_k$ and proceed by induction on $k$.


By Theorem~\ref{thm:mut-lb} we get that the assertion holds for $k=2$: $S_2$ is a feasible total mutual-visibility set of $H_2$. 
Let us assume it is true for $H_{k}$, $k\ge 2$, and consider $H_{k+1} = H_{k} \boxtimes G_{k+1}$. By the inductive hypothesis, $S_{k}$ is a feasible total mutual-visibility set of $H_{k}$. By the proof of Theorem~\ref{thm:mut-lb}, $S_{k+1}$ is a feasible total mutual-visibility set of $H_{k+1}$. Thus
\[
\begin{array}{rcl}
   \mut(H_{k+1}) & \ge &  n(H_k) n(G_{k+1}) -  (n(H_k) - \mut(H_k))(n(G_{k+1}) - \mut(G_{k+1}))  \\
            & \ge & n(H_k) n(G_{k+1}) - \\
            & & \left( n(H_k) - \left( \prod_{i=1}^{k} n(G_i) - \prod_{i=1}^{k} (n(G_i) - \mut(G_i))\right)\right)(n(G_{k+1}) - \mut(G_{k+1})) \\
& = & n(H_{k+1}) - \left(\prod_{i=1}^k (n(G_i) - \mut(G_i))\right)(n(G_{k+1}) - \mut(G_{k+1})) \\
& = & n(H_{k+1}) - \prod_{i=1}^{k+1} (n(G_i) - \mut(G_i))
\end{array}
\]
and we are done.
\qed
\end{proof}

The following result (cf. Theorem~\ref{thm:path-product-multi}) shows that there are $(\mu, \mut)$-graphs for which the lower bound provided by~\eqref{eq:mut-lb-perfect} coincides with the mutual-visibility number of the strong product. Notice that it concerns the strong product of paths with at least three vertices, whereas Theorem~\ref{thm:P2-block-strong} (cf. Section~\ref{sec:strong-prisms} where strong prisms are considered) will provide the exact value of $\mu(P_2\boxtimes G)$ for every block graph $G$ (and hence also $\mu(P_2\boxtimes P_n)$ with $n\ge 2$).

We first recall the following result that uses convex hulls to provide an upper bound to $\mu(G)$.

\begin{lemma}\label{lem:hull}
{\rm \cite[Lemma 2.3]{DiStefano-2022} }
Given a graph $G$, let $V_1,\ldots,V_k$ be subsets of $V(G)$ such that
$\bigcup_{i=1}^k V_i = V(G)$. Then, $\mu(G) \le \sum_{i=1}^k \mu( \Hull(V_i) )$.
\end{lemma}

\begin{theorem}
\label{thm:path-product-multi}
If $H_k = P_{n_1}\boxtimes \cdots \boxtimes P_{n_k}$, where $k\ge 2$ and  $n(P_{n_i})\ge 3$ for $i\in [k]$, then
$$\mu(H_k) = \prod_{i=1}^k n(P_{n_i}) - \prod_{i=1}^k (n(P_{n_i}) - 2).$$
\end{theorem}

\begin{proof}
Let $X_i\subseteq V(P_{n_i})$ be the (total) mutual-visibility set of $P_{n_i}$ formed by the end-vertices of the path. Note that $X_i$ is feasible. According to the proof of Corollary~\ref{cor:mut-lb-recursive}, we get that
\begin{equation}\label{eq:Sk}
S_k = (V(P_{n_1})\times \cdots \times V(P_{n_k})) \setminus
           (\overline{X_i} \times \cdots \times \overline{X_i})
\end{equation}
is a total mutual-visibility set of $H_k$. By the same corollary, we also get the following lower bound:
$$\mu(H_k) \ge \mut(H_k)
           \ge \prod_{i=1}^k n(P_{n_i}) - \prod_{i=1}^k (n(P_{n_i}) - \mut(P_{n_i}))
             = \prod_{i=1}^k n(P_{n_i}) - \prod_{i=1}^k (n(P_{n_i}) - 2) .$$
Let the tuple $(\ell_1,\ldots,\ell_k)$ denote the generic vertex of $H_k$, where $\ell_i\in [n(P_{n_i})]$, $i\in [k]$. We define the following two subsets of $V(H_k)$:

\begin{itemize}
\item
$\Int = \{ (\ell_1,\ldots,\ell_k):\ \forall~ i\in [k], \ell_i\neq 1 \mbox{ and } \ell_i\neq n(P_{n_i}) \}$;
\item
$\Ext = \{ (\ell_1,\ldots,\ell_k):\ \exists~ i\in [k], \ell_i= 1 \mbox{ or } \ell_i = n(P_{n_i}) \}$.
\end{itemize}
From these definitions, it can be easily observed that $\Int$ and $\Ext$ form a partition of $V(H_k)$. Moreover, according to this notation, we get the following characterization of the total mutual-visibility set $S_k$ defined in~\eqref{eq:Sk}:
\begin{equation}\label{eq:SkExt}
S_k=\Ext.
\end{equation}
To prove the upper bound for $\mu(H_k)$, we use Lemma~\ref{lem:hull}. To this end, we determine the (minimum) number of induced and convex \emph{diagonals} which cover all the vertices of $H_k$. A diagonal is either \emph{degenerated} or \emph{non-degenerated}: non-degenerated diagonals are paths of $H_k$ formed by at least two vertices and having the form $((\ell_1,\ldots,\ell_k), (\ell_1+1,\ldots,\ell_k+1), (\ell_1+2,\ldots,\ell_k+2), \ldots)$, whereas each degenerated diagonal consists of a single vertex. These two kinds of diagonals are formally defined as follows (see. Figure~\ref{fig:diagonals-in-3D} for two examples):
\begin{description}
\item[$(i)$]
Each vertex in $I= \{ (\ell_1,\ldots,\ell_k):\  \exists~ i\in [k], \ell_i=1 \mbox{ and } \forall j\in [k], \ell_j\neq n(P_{n_j}) \}$ belongs to non-degenerated diagonals. In particular, each vertex in $I$ is the \emph{initial vertex} (i.e., one of its end-vertices) of such kind of diagonals.
\item[$(ii)$]
If $(\ell_1,\ldots,\ell_k)$ belongs to a non-degenerated diagonal $d$, then also its neighbor $(\ell_1+1,\ldots,\ell_k+1)$ (if it exists in $H_k$) belongs to $d$.  This property allows to define non-degenerated diagonals, along all their maximal length, till some \emph{terminating vertex} having at least one coordinate $\ell_i$ such that $\ell_i = n(P_{n_i})$. We denote by $T$ all the terminating vertices of non-degenerated diagonals.
\item[$(iii)$]
The set $D= \{ (\ell_1,\ldots,\ell_k) ~:~ \exists~ i,j\in [k], \ell_i=1 \textit{ and } \ell_j = n(P_{n_j}) \}$ contains all vertices forming degenerated diagonals.
\end{description}

\begin{figure}[ht!]
\begin{center}
\includegraphics[height=4.2cm]{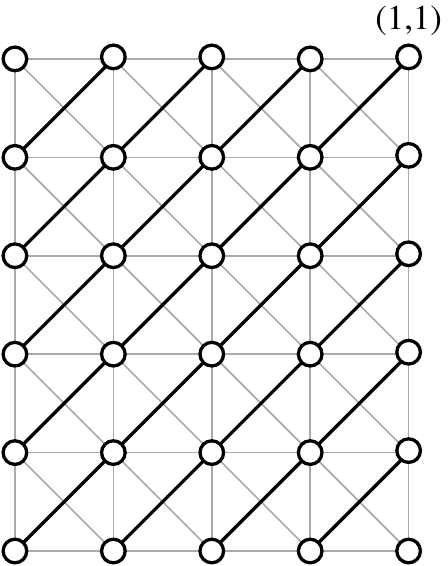}
\hspace{10mm}
\includegraphics[height=7.5cm]{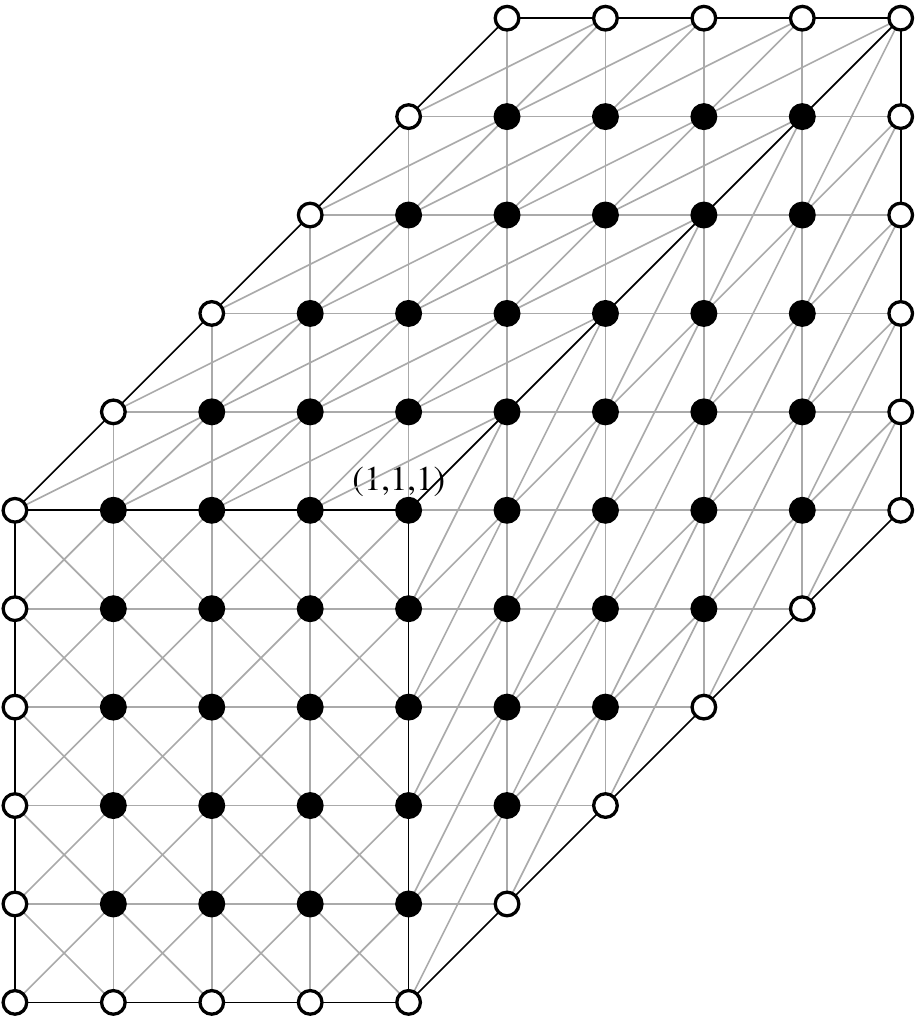}
\caption{\small
Visualization of diagonals as defined in the proof of Theorem~\ref{thm:path-product-multi}. (\emph{Left}) In this strong product $H_2 = P_5\boxtimes P_6$, the thicker and bolder lines represent non-degenerated diagonals.
(\emph{Right}) A representation of $H_3 = P_5\boxtimes P_6\boxtimes P_6$ as an ``opaque rectangular cuboid'' where the position of the vertex with coordinates $(1,1,1)$ is shown. Black vertices represent the elements of set $I$, that is the starting point of non-degenerate diagonals; white vertices represent the elements of set $D$, that is vertices forming degenerated diagonal. All such diagonals cover the whole graph $H_3$. }
\label{fig:diagonals-in-3D}
\end{center}
\end{figure}

\noindent
Notice that the non-degenerated diagonals are pairwise vertex disjoint. The requested covering of $H_k$ is given by all the maximal non-degenerated diagonals along with all the degenerated diagonals. Now, let $X\subseteq V(H_k)$ be the set containing the end-vertices of each non-degenerated diagonal and all the vertices forming degenerated diagonals; formally, $X=I \cup T \cup D$. According to Lemma~\ref{lem:hull} we know that $\mu(H_k)\le |X|$. By Eq.~\ref{eq:SkExt}, we complete the proof by showing that $\Ext = X$.

\begin{itemize}
\item
Let $v=(\ell_1,,\ldots,\ell_k)\in \Ext$. By definition of $\Ext$, there exists a coordinate $\ell_i$ of $v$ for which $\ell_i = 1$ or $\ell_i= n(P_{n_i})$.
If both $\ell_i = 1$ and $\ell_i= n(P_{n_i})$ hold, then property $(iii)$ in the definition of diagonals holds. This means that $v\in D$ and hence $v\in X$.
If $\ell_i = 1$ and $\ell_i\neq n(P_{n_i})$ hold, then property $(i)$ in the definition of diagonals holds. This means that $v\in I$ and hence $v\in X$.
If $\ell_i \neq 1$ and $\ell_i= n(P_{n_i})$ hold for each $i$, then consider the smallest coordinate $\ell_j$ of $v$. According to  property $(ii)$, the vertex $v'= (\ell_1-(\ell_j-1),\ldots,\ell_k-(\ell_j-1))$ lies in the set $I$ from which a non-degenerated diagonal starts. This implies that $v\in T$ and hence $v\in X$.
So, in all cases, we have $v\in X$.
\item
Let $v=(\ell_1,\ldots,\ell_k)\in X$. If $v\in I\cup D$, then $v\in \Ext$ trivially holds. Assume now $v\in T$, that is $v$ is the end-vertex of a non-degenerated diagonal $d$ starting at some vertex $v'=(\ell_1',\ldots,\ell_k')$ for which $(i)$ holds, and made maximal by iteratively applying property $(ii)$. According to $(ii)$, an end-vertex of $d$ must be in $\{ n(P_{n_1}), \ldots, n(P_{n_k})\}$, and hence $v\in \Ext$.
\end{itemize}
This proves that $\Ext = X$ holds.
\qed
\end{proof}

It seems worth pointing out that the result of Theorem~\ref{thm:path-product-multi} for two- and three-dimensional strong grids reads as follows:
\begin{align*}
\mu(P_{n_1}\boxtimes P_{n_2}) & = 2n_1 + 2n_2 - 4\,,\\
\mu(P_{n_1}\boxtimes P_{n_2}\boxtimes P_{n_3}) & = 2(n_1n_2 + n_1n_3 + n_2n_3) - 4(n_1 + n_2 + n_3) + 8\,.
\end{align*}

To conclude the analysis, notice there are examples of graphs for which the bound of Theorem~\ref{thm:mut-lb} is not sharp. An example of this situation is given in Figure~\ref{fig:no_tree_solution}.

\begin{figure}[ht!]
\begin{center}
\includegraphics[height=7.5cm]{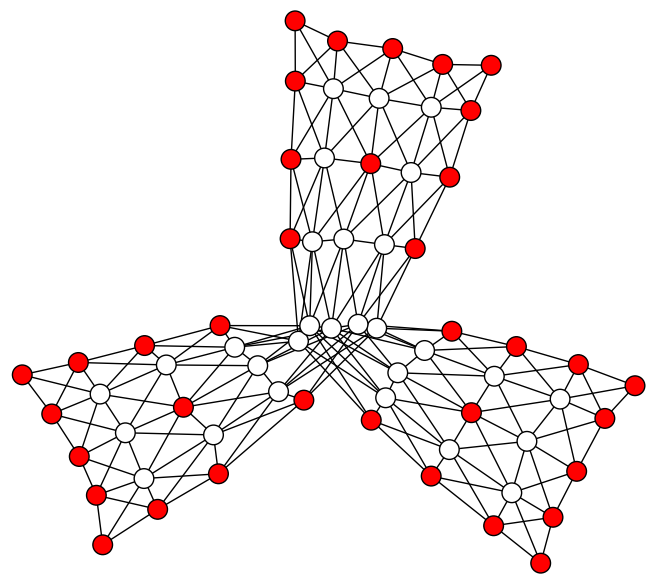}
\caption{\small The graph $T\boxtimes P_5$ and its mutual-visibility set of cardinality $36$.}
\label{fig:no_tree_solution}
\end{center}
\end{figure}

Let $T$ be the tree obtained from $K_{1,3}$ by subdividing each of its edges three times. Since both $T$ and $P_5$ admit feasible $\mut$-sets, 
Theorem~\ref{thm:mut-lb} implies $\mu(T\boxtimes P_5)\ge \mu(P_5)n(T) + \mu(T)n(P_5) - \mu(P_5)\mu(T) = 35$, but in Figure~\ref{fig:no_tree_solution} we can see a mutual-visibility set of cardinality $36$ found by computer search. This example also shows that even when both factors of a strong product are $(\mu, \mut)$-graphs, their strong product does not achieve the equality in the bound of Theorem~\ref{thm:mut-lb}. Note that this particular example can be generalized to an infinite family of graphs where the difference between the mutual-visibility number and the bound of the theorem becomes arbitrarily large. This situation also suggests that generalizing Theorem~\ref{thm:path-product-multi} (when $k=2$) to the strong product of two arbitrary trees might be a challenging problem.

\begin{corollary}\label{cor:universal}
If $G_1, \ldots, G_k$ are non-complete graphs, each containing a universal vertex, then
$$\mut(G_1\boxtimes \cdots \boxtimes G_k) =  \prod_{i=1}^k n(G_i) - 1\,.$$
\end{corollary}
\begin{proof}
Observe that if $v$ is a universal vertex of $G_i$, then $V(G_i)\setminus \{v\}$ is a feasible $\mut$-set of $G_i$. Hence, since each $G_i$ is a $(\mu, \mut)$-graph that admits a feasible $\mut$-set, by Corollary~\ref{cor:mut-lb-recursive} we get $\mut(G_1\boxtimes \cdots \boxtimes G_k) \ge \prod_{i=1}^k n(G_i) - 1$. The claim follows by simply observing that $G_1\boxtimes \cdots \boxtimes G_k$ is not a clique and hence $\mu(G_1\boxtimes \cdots \boxtimes G_k) < \prod_{i=1}^k n(G_i)$.
\qed
\end{proof}

We conclude the section with another lower bound on $\mu(G\boxtimes H)$ in terms of the mutual-visibility number of the factors.

\begin{theorem}
\label{thm:mu-product}
If $G$ and $H$ are graphs, then
$$\mu(G\boxtimes H) \ge \mu(G)\mu(H)\,.$$
\end{theorem}

\proof
Let $S_G$ be a $\mu$-set of $G$ and $S_H$ be a $\mu$-set of $H$. Then we claim that $S = S_G \times S_H$ is a mutual-visibility set of $G\boxtimes H$.

Let $(g,h)$ and $(g',h')$ be arbitrary two vertices from $S$. Since $g,g'\in S_G$, there exists a shortest $g,g'$-path $P_G$ in $G$ such that no internal vertex of $P_G$ is in $S_G$. Let the consecutive vertices of $P_G$ be $g=g_0, g_1, \ldots, g_k=g'$. Similarly, there is a shortest $h,h'$-path $P_H$ in $H$ such that no internal vertex of $P_H$ is in $S_H$. Let the consecutive vertices of $P_H$ be $h=h_0, h_1, \ldots, h_\ell=h'$. Note that it is possible that $k=0$ or $\ell = 0$ (but not both). Assume without loss of generality that $\ell\le k$. Then the vertices
$$(g,h) = (g_0,h_0), (g_1,h_1), \ldots, (g_\ell, h_\ell), (g_{\ell+1}, h_\ell)), \ldots, (g_k, h_\ell)) = (g',h')$$
induce a shortest $(g,h),(g',h')$-path $Q$ in $G\boxtimes H$. Clearly, no internal vertex of $Q$ is in $S$, hence we conclude that $S$ is a mutual-visibility set.
\qed

\section{Mutual-visibility in strong prisms}\label{sec:strong-prisms}

In this section we study the mutual-visibility number of strong prisms, that is graphs in the form $G\boxtimes P_2$. We begin with the following general lower bound.

\begin{theorem}
\label{thm:P2-lower}
If $G$ is a graph, then $\mu(G\boxtimes P_2) \ge \max\{n(G), 2\mu(G)\}$.
\end{theorem}
\begin{proof} 
%
Since $\mu(P_2)=2$, by Theorem~\ref{thm:mu-product} we get $\mu(G\boxtimes P_2) \ge 2\mu(G)$. Assuming $V(P_2)=\{p,q\}$, we prove the statement by showing that $S = V(G) \times \{p\}$ is a mutual-visibility set of $G\boxtimes P_2$.

Let $(g,p)$ and $(g',p)$, with $g\neq g'$, be arbitrary two distinct vertices from $S$. Consider a shortest $g,g'$-path $P_G$ in $G$. Let the consecutive vertices of $P_G$ be $g=g_0, g_1, \ldots, g_k=g'$. Since $g\neq g'$ we get $k\ge 1$. If $k=1$, then $(g,p)$ and $(g',p)$ are connected and there is nothing to prove.  If $k\ge 2$, then the vertices
$$(g,p) = (g_0,p), (g_1,q), \ldots, (g_{k-1},q), (g_k, p) = (g', p) $$
induce a shortest $(g,p),(g',p)$-path $Q$ in $G\boxtimes P_2$. Clearly, no internal vertex of $Q$ is in $S$, hence we conclude that $S$ is a mutual-visibility set of $G\boxtimes P_2$.
%
\qed
\end{proof}

Theorem~\ref{thm:P2-lower} can be improved for $(\mu, \mut)$-graphs as we next show.

\begin{theorem}\label{thm:P2-G-lower}
If $G$ is a  $(\mu, \mut)$-graph that admits a feasible $\mut$-set, then $\mu(G\boxtimes P_2) \ge \mu(G) + n(G)$.
\end{theorem}
\begin{proof}
Consider a feasible $\mut$-set (which is also a $\mu$-set) of $G$ and a total mutual-visibility set $X$ for $P_2$ composed by only one vertex. If $G$ is trivial, then the statement clearly holds, otherwise we can apply the first inequality of Theorem~\ref{thm:mut-lb} and~\eqref{eq:mut-lb-perfect} as follows:
$$
   \begin{array}{rcl}
   \mu(G\boxtimes P_2) & \ge &  \mu(G)n(P_2)   + |X|n(G) -  \mu(G)|X| \\
                      & = & \mu(G)\cdot 2   + n(G) -  \mu(G) \\
                      & = & \mu(G)   + n(G).
   \end{array}$$ \qed
\end{proof}

Next we show that the lower bound of Theorem~\ref{thm:P2-G-lower} is attained by block graphs (in particular non-complete block graphs, which admit feasible $\mut$-sets). To do so, we need the following lemma.

\begin{lemma}
\label{lem:cut}
Let $x$ be a cut vertex of a graph $G$. Then there exists a $\mu$-set of $G\boxtimes P_2$ which contains at most one copy of $x$ in the two $G$-layers.
\end{lemma}

\begin{proof}
Let $S$ be a $\mu$-set of $G\boxtimes P_2$ and suppose that $x', x''\in S$, where $x'$ and $x''$ are the copies of $x$ in the $G$-layers. Let $H$ and $H'$ be two components of  
$(G\boxtimes P_2) - \{x',x''\}$. Then $S\cap V(H) = \emptyset$ or $S\cap V(H') = \emptyset$, say $S\cap V(H) = \emptyset$, for otherwise $S$ is not even a mutual-visibility set. Now the set $S' = (S \cup \{z\})\setminus \{x'\}$, where $z$ is an arbitrary vertex of $H$, is also a $\mu$-set of $G\boxtimes P_2$.
\qed
\end{proof}

\begin{theorem}
\label{thm:P2-block-strong}
If $G$ is a block graph, then $\mu(G\boxtimes P_2) = n(G) + \mu(G)$.
\end{theorem}
\begin{proof}
Let $X$ be the set of the cut vertices of $G$. By Lemma~\ref{lem:cut} there exists a $\mu$-set $S$ with at most one copy of each vertex in $X$. We show that $S$ includes one copy of any vertex $v$ of $G$ if $v$ is a cut vertex and two copies of $v$ otherwise. This proves the statement.

 Consider two vertices $u,v$ of the same copy $G'$ of $G$ and the shortest $u,v$-path in $G$. If $u$ and $v$ belong to the same block then they are adjacent since a block is a clique by definition. Otherwise, consider the shortest $u,v$-path in $G'$: it is unique and passes only through cut vertices of $G'$. Since for each vertex in $X$ only one copy is in $S$, there exists a shortest $u,v$-path in $G\boxtimes P_2$ without internal vertices in $S$. Then $u$ and $v$ are $X$-visible.

Assume that $u$ and $v$ do not belong to the same copy of $G$. If they belong to two copies of the same block, then they are adjacent. Otherwise, as above, there exists a shortest $u,v$-path in $G\boxtimes P_2$ passing through copies of cut vertices of $G$ and without internal vertices in $S$.
\qed
\end{proof}

We conclude the paper by demonstrating the sharpness of the bound of Theorem~\ref{thm:P2-G-lower}.

\begin{theorem}
\label{thm:P2-cycle-strong}
If $n\ge 3$, then
$$
 \mu(C_n\boxtimes P_2) =
 	\begin{cases}
  		6; &  n\in\{3,4,5\},\\
  		7; &  n=6, \\
  		n; &  n\ge 7. \\
 	\end{cases}
$$
\end{theorem}
\begin{proof}
Recall from~\cite{DiStefano-2022} that $\mu(C_n)= 3$, $n\geq 3$. Hence, Theorem~\ref{thm:P2-G-lower} implies that $\mu(C_n\boxtimes P_2) \ge 6$ when $n\le 6$ and $\mu(C_n\boxtimes P_2) \ge n$ for $n\ge 6$.

We checked by computer that $\mu(C_n\boxtimes P_2) = 6$ when $n\leq 5$ and $\mu(C_6\boxtimes P_2) = 7$.
 Assume in the rest that $n > 6$ which means that $\mu(C_n\boxtimes P_2) \ge n$. Let $S$ be a $\mu$-set of $C_n\boxtimes K_2$. We need to show that $|S| \le n$.

Let $v_0,v_1,\ldots, v_{n-1}$ and $v'_0,v'_1,\ldots, v'_{n-1}$ be the vertices of the two $C_n$-layers. Then a pair $v_i, v'_i$ is called a \emph{separating pair}. $S$ cannot contain three separating pairs since $|S| \ge n\ge 7$ for otherwise a vertex from $S$ which is not in three fixed separating pairs cannot be in visibility with all the vertices in the separating pairs. Hence $|S| \leq n+2$. If $|S| \in \{n+1, n+2\}$, then consider one separating pair $v_j, v'_j$. Then there exists a vertex in $\{v_{j-1}, v'_{j-1}\} \cap S$ and a vertex in $\{v_{j+1}, v'_{j+1}\} \cap S$ which are not $S$-visible. We conclude that $|S| \le n$.
\qed
\end{proof}

\section{Concluding remarks and future work}\label{sec:conclusion}

This work suggests some further research directions.
We have shown that block graphs and certain cographs are all $(\mu, \mut)$-graphs. Notice that cographs can be generated by using true and false twins, and that block graphs can be generated by using true twins and pendant vertices. A superclass of both cographs and block graphs is that formed by \emph{distance-hereditary graphs}. In fact, these graphs can be generated by using true twins, false twins, and pendant vertices. It would be interesting to characterize all the distance-hereditary graphs that are $(\mu, \mut)$-graphs. We left open the general question about characterizing the larger class $\mathcal{G}$ of graphs formed by $(\mu, \mut)$-graphs. In addition, another characterization that would be of interest concerns finding all graphs $G$ for which $\mut(G)=0$.

Concerning specific results, in view of Theorem~\ref{thm:path-product-multi} (when we consider $k=2$), it would be interesting to study $\mu(T\boxtimes T')$ for any two trees $T$ and $T'$. Also, Theorem~\ref{thm:P2-G-lower} provides the lower bound  $\mu(G\boxtimes P_2) \ge \mu(G) + n(G)$ for each $(\mu, \mut)$-graph $G$, whereas Theorem~\ref{thm:P2-block-strong} states that the equality is attained in the case of block graphs. We wonder if this equality holds for each $(\mu, \mut)$-graph.

Another interesting point is studying other possible variations of the general concept of mutual-visibility sets and their relationships, as well as relationships with the concept of general position sets.

It took a long time before we were able to produce a revised version of the present paper. As we finalise it, we would just like to add that in the meantime, the concept of total mutual-visibility has already given rise to several further studies, largely inspired by the original concluding remarks above.

\section*{Acknowledgments}

We are very grateful to the reviewers who identified the error in the previous version of Theorem 4.1 and also for correcting one of the cases in Theorem 5.5.

\smallskip
S. Cicerone and G. Di Stefano were partially supported by the European Union - NextGenerationEU
under the Italian Ministry of University and Research (MUR) National Innovation Ecosystem grant ECS00000041 - VITALITY - CUP J97G22000170005, and by the Italian National Group for Scientific Computation (GNCS-INdAM).
S. Klav\v{z}ar was partially supported by the Slovenian Research Agency (ARRS) under the grants P1-0297, N1-0218, N1-0285. I. G. Yero has been partially supported by the Spanish Ministry of Science and Innovation through the grant PID2019-105824GB-I00. Moreover, this investigation was partially developed while I. G. Yero was visiting the University of Ljubljana supported by  ``Ministerio de Educaci\'on, Cultura y Deporte'', Spain, under the ``Jos\'e Castillejo'' program for young researchers (reference number: CAS21/00100).


\end{document}